\documentclass[11pt,a4paper]{article}

\usepackage[latin1]{inputenc}
\usepackage{amsmath,amsfonts,amssymb,wasysym}
\usepackage{enumerate}
\usepackage{graphicx,epsf,color}
\usepackage{makeidx}
\usepackage[normalem]{ulem}
\usepackage{subfigure}
\usepackage{amsthm}
\usepackage{caption}
\usepackage{mathrsfs} 
\usepackage{float}
\usepackage{placeins} 
\usepackage{geometry} 
\usepackage{multicol}
\usepackage{authblk}

\newtheorem{theorem}{Theorem}[section]
\newtheorem{proposition}[theorem]{Proposition}
\newtheorem{lemma}[theorem]{Lemma}
\newtheorem{corollary}[theorem]{Corollary}
\newtheorem{defn}[theorem]{Definition}

\newtheorem{remark}[theorem]{Remark}

\newcommand{\RR}{\mathbb R}
\newcommand{\ZZ}{\mathbb Z}
\newcommand{\CC}{\mathbb C}

\newcommand{\KK}{\mathbb K}

\newcommand{\om}{\omega}

\newcommand{\rank}{\operatorname{rank}}

\newcommand{\im}{\operatorname{Im}}

\def\df={\buildrel {\rm def}\over =}

\begin{document}

\title{{\LARGE {Symplectic singularity of curves with semigroups} {\Large $(4,5,6,7), (4,5,6)$} and \Large $(4,5,7)$ }}
\author{Fausto Assun\c{c}\~{a}o de Brito Lira \thanks{F. Lira has been supported by CAPES, CNPq grant no. 245309/2012-8 and Fapesp grant no. 2012/16426-4.}, Wojciech Domitrz \thanks{W. Domitrz was partially supported by NCN grant no. DEC-2013/11/B/ST1/03080.} \\ and Roberta Wik Atique \thanks{Wik-Atique was partially supported by CNPq grant no. 245309/2012-8 and Fapesp grant no. 2012/16426-4.}}

\maketitle


{\footnotesize{\textbf{Abstract}: We study the local symplectic algebra of curves with semigroups $(4,5,6,7)$, $(4,5,6)$ and $(4,5,7)$.
We use the method of algebraic restrictions to parameterized curves as in \cite{D1}. A new discrete invariant for algebraic restrictions to parameterized quasi-homogeneous curves is introduced. This invariant together with the method of algebraic restriction can distinguish different symplectic orbits of quasi-homogeneous curves.}}


\section{Introduction}

This paper is concerned with the problem of symplectic classification of parameterized quasi-homogeneous curve-germs in the symplectic space $(\KK^{2n},\om)$ $(\KK=\RR \, \text{or} \, \CC)$.

We say that two parameterized curve-germs $f,g:(\KK,0)\to (\KK^{2n},0)$ are \textbf{symplectically equivalent}, or  \textbf{symplectomorphic}, if there exist a germ of diffeomorphism $\phi:(\KK,0)\to (\KK,0)$ and a germ of symplectomorphism $\Phi:((\KK^{2n},\om),0))\to ((\KK^{2n},\om),0)$ such that $\Phi \circ f=g \circ \phi$. By  symplectomorphism we mean a diffeomorphism  which preserves the symplectic form $\om$. We study the action of a subgroup of the classical $\mathcal A$-group where the diffeomorphism in the target is a symplectomorphism.

The problem of symplectic classification of singular varieties was introduced by V.I. Arnold in \cite{A1}. He proved that the $A_{2k}$-singularity of a planar curve (the orbit with respect to the standard $\mathcal A$-equivalence of parameterized curves) splits into exactly $2k+1$ symplectic singularities.
He distinguish the orbits by the orders of tangency of the parameterized curve to the nearest smooth Lagrangian submanifold. Arnold posed a problem of expressing these new symplectic invariants in terms of the local algebra's interaction with the symplectic structure and he proposed to call this interaction the \textbf{local symplectic algebra}.

In \cite{IJ} G. Ishikawa and S. Janeczko classified symplectic singularities of curves in the 2-dimensional symplectic space. A symplectic form on a 2-dimensional manifold is a special case of a volume form on a smooth manifold. The generalization of results in \cite{IJ} to volume-preserving classification of singular varieties and maps in arbitrary dimensions was obtained in \cite{DR}.

A symplectic singularity is stably simple if it is simple and remains simple if the ambient symplectic space is symplectically embedded (i.e. as a symplectic submanifold) into a larger symplectic space. In \cite{K} P. A. Kolgushkin classified the stably simple symplectic singularities of curves (in the $\CC$-analytic category).

A generalization of the Darboux-Givental Theorem (\cite{AG}) to germs of  quasi-homogeneous subsets of the
symplectic space was obtained in \cite{DJZ2} and reduces the problem of symplectic classification of germs of
quasi-homogeneous subsets to the problem of classification of algebraic restrictions of symplectic
forms to these subsets. By this method, complete symplectic classifications
of the $A-D-E$ singularities of planar curves and the $S_5$ singularity were obtained in \cite{DJZ2}.

In \cite{D1} W. Domitrz used the method of algebraic restrictions to study the problem of symplectic classification of quasi-homogeneous parametrized curve-germs. He proved that the vector space of algebraic restrictions of closed 2-forms to the germ of a $\KK$-analytic curve is a finite-dimensional vector space. The  symplectic orbits were distinguished by the symplectic invariants: symplectic multiplicity, index of isotropy and the Lagrangian tangency order.

The method of algebraic restrictions was also applied to the zero-dimensional symplectic isolated complete intersection singularities (see \cite{D2}) and to other 1-dimensional isolated complete intersection singularities: the $S_\mu$ symplectic singularities for $\mu > 5$ in \cite{DT1}, the $T_7-T_8$ symplectic singularities in \cite{DT2}, the $W_8-W_9$ symplectic singularities in \cite{T1} and the $U_7, U_8$ and $U_9$ symplectic singularities in \cite{T2}.

In this paper we use the method of algebraic restrictions described in \cite{D1} to obtain the symplectic classification of all parametrized curve-germs with semigroups $(4,5,6,7), (4,5,6)$ and $(4,5,7)$.
The choice of the semigroup $(4,5,6,7)$ is due to the following reason: this semigroup shows us that the system of symplectic invariants used in \cite{D1} is not enough to distinguish all the symplectic orbits here obtained. 
Then we apply this method to others semigroups as examples.

The contribution of this paper is to introduce a new discrete invariant for algebraic restrictions of $k$-forms to parameterized quasi-homogeneous curves: \textbf{the proportional minimum quasi-degree part}. This is a very natural invariant for algebraic restrictions but the translation of this invariant for the symplectic action is still open.

We  present an algorithm in Singular software \cite{DGPS} to find a basis of the space of algebraic restrictions of closed 2-forms to a germ of parameterized curve.

This paper is organized as follows: In section 2 we recall the method of algebraic restrictions for quasi-homogeneous curves following \cite{D1}. The discrete symplectic invariants appear in section 3. Finally, in section 4, we present the classification.


\section{The method of algebraic restrictions of quasi-homogeneous curves}

In this section we present the method of algebraic restrictions. More details can be found in \cite{DJZ2} and \cite{D1}.

\begin{defn}
A curve-germ $f:(\KK,0)\to (\KK^{m},0)$ is \textbf{quasi-homogeneous} if there exist coordinate systems $t$ on $(\KK ,0)$ and $(x_1,\ldots,x_m)$ on $(\KK^{m},0)$ and positive integers $(\lambda_1,\ldots,\lambda_m)$ such that
\[
df \left(t \frac{d}{dt} \right)=E\circ f,
\]
where $E=\sum_{i=1}^{m}\lambda_i x_i \partial / \partial x_i$ is the germ of the Euler vector field on $(\KK^{m},0)$. The coordinate system $(x_1,\ldots,x_m)$ is  also called quasi-homogeneous, and the numbers $(\lambda_1,\ldots,\lambda_m)$ are called \textbf{weights}.
\end{defn}

\begin{defn}
Positive integers $\lambda_1,\ldots ,\lambda_m$ are \textbf{linearly dependent over non-negative integers} if there exist $j$ and non-negative integers $k_i$ for $i\neq j$ such that $\lambda_j=\sum_{i\neq j} k_i\lambda_i$. Otherwise we say that $\lambda_1,\ldots,\lambda_m$ are \textbf{linearly independent over non-negative integers}.
\end{defn}

It is not difficult to show that quasi-homogeneous curves have the following form in quasi-homogeneous coordinates.

\begin{proposition}\label{a-eq}
A curve-germ $f$ is quasi-homogeneous if and only if it is $\mathcal A$-equivalent to
\[
t\to (t^{\lambda_1},\ldots,t^{\lambda_k},0,\ldots,0),
\]
where $\lambda_1 < \cdots < \lambda_k$ are positive integers linearly independent over non-negative integers.
\end{proposition}

The integers $\lambda_1,\ldots,\lambda_k$ generate the semigroup of the curve $f$, which we denote by $(\lambda_1,\ldots,\lambda_k)$.


\begin{defn}
The germ of a function, of a differential $k$-form, or a vector field $\alpha$ on $(\KK^m,0)$ is \textbf{quasi-homogeneous} in a coordinate system $(x_1,\ldots,x_m)$ on $(\KK^m,0)$ with positive weights $(\lambda_1,\ldots,\lambda_m)$ if ${\mathcal L}_E\alpha=\delta \alpha$, where $E=\sum_{i=1}^{m}\lambda_ix_i\partial / \partial x_i$ is  the germ of the Euler vector field on $(\KK^m,0)$ and $\delta$ is a real number called the \textbf{quasi-degree}.
\end{defn}

It is easy to show that $\alpha$ is quasi-homogeneous in a coordinate system $(x_1,\ldots,x_m)$ with weights $(\lambda_1,\ldots,\lambda_m)$ if and only if $F_{t}^{*}\alpha=t^{\delta}\alpha$, where $F_{t}(x_1,\ldots,x_m)=(t^{\lambda_1}x_1,\ldots,t^{\lambda_m}x_m)$. Thus germs of quasi-homogeneous functions of quasi-degree $\delta$ are germs of weighted homogeneous polynomials of degree $\delta$. The coefficient $f_{i_1,\ldots,i_k}$ of the quasi-homogeneous differential $k$-form $\sum f_{i_1,\ldots,i_k}dx_{i_1}\wedge \cdots \wedge dx_{i_k}$ of quasi-degree $\delta$ is a weighted homogeneous polynomial of degree $\delta - \sum_{j=1}^{k}\lambda_{i_j}$. The coefficient $f_i$ of the quasi-homogeneous vector field $\sum_{i=1}^{m}f_i\partial / \partial x_i$ of quasi-degree $\delta$ is a weighted homogeneous polynomial of degree $\delta + \lambda_i$.



\begin{proposition}\label{prop2.5}(\cite{DJZ2})
If $X$ is the germ of a quasi-homogeneous vector field of quasi-degree $i$ on $(\KK^{m},0)$ and $\om$ is the germ of a quasi-homogeneous differential $k$-form of quasi-degree $j$ on $(\KK^{m},0)$ then ${\mathcal L}_X\om$ is the germ of a quasi-homogeneous differential $k$-form of quasi-degree $i+j$ on $(\KK^m,0)$.
\end{proposition}

Given the germ of a smooth manifold $(M,p)$ denote by ${\Lambda}^{k}(M)$ the space of all germs at $p$ of differential $k$-forms on $M$. Given a curve-germ $f:(\KK,0)\to (M,p)$ we introduce the following subspaces of ${\Lambda}^{k}(M)$:
\[
\begin{array}{c}
\Lambda_{\im(f)}^{k}(M)=\{\om\in {\Lambda}^k (M): \om(f(t))=0\,\ \text{for any}\,\ t\in (\KK,0)\},\vspace{0.2cm}\\
{\mathscr A}_{0}^{k}(Im(f),M)=\{\alpha + d\beta:\,\ \alpha \in {\Lambda}_{\im(f)}^{k}(M),\, \beta \in {\Lambda}_{\im(f)}^{k-1}(M)\}.\\
\end{array}
\]

\begin{defn}
Let $f:(\KK,0)\to (M,p)$ be a curve-germ and  $\om \in {\Lambda}^{k}(M)$. The \textbf{algebraic restriction} of $\om$ to $\im(f)$ is the equivalence class of $\om$ in ${\Lambda}^{k}(M)$, where the equivalence is as follows: $\om$ is equivalent to $\tilde{\om}$ if $\om - \tilde{\om}\in \mathscr{A}_{0}^{k}(Im(f),M)$. We denote the algebraic restriction of $\om$ to $\im(f)$ as $[\om]_f$.
\end{defn}


The next result provides some basic properties about the set of zero algebraic restrictions to a parameterized curve-germ.

\begin{proposition}(\cite{DJZ2})\label{propalgres}
If $\om\in {\mathscr A}_{0}^{k}(\im(f),M)$ then $d\om\in {\mathscr A}_{0}^{k+1}(\im(f),M)$ and $\om\wedge \alpha \in {\mathscr A}_{0}^{k+p}(\im(f),M)$ for any $p$-form $\alpha$ on $M$.
\end{proposition}

In particular, the exterior derivative and the wedge product for algebraic restriction to a parametrized curve-germ are well defined.

Let $(M,p)$ and $(\widetilde{M},\tilde{p})$ be germs of smooth equi-dimensional manifolds. Let $f:(\KK,0)\to (M,p)$ and $\tilde{f}:(\KK,0)\to (\tilde{M},\tilde{p})$ be curve-germs. Let $\om$ and $\tilde{\om}$ be germs of $k$-forms on $(M,p)$ and $(\widetilde{M},\tilde{p})$, respectively.

\begin{defn}
Two algebraic restrictions $[\om]_f$ and $[\tilde{\om}]_{\tilde{f}}$ are \textbf{diffeomorphic} if there exist germs of diffeomorphisms $\Phi :(\widetilde{M},\tilde{p})\to (M,p)$ and $\phi:(\KK,0)\to(\KK,0)$ such that $\Phi\circ \tilde{f}\circ \phi^{-1}=f$ and ${\Phi}^{*}([\om]_f):=[{\Phi}^{*}\om]_{{\Phi}^{-1}\circ f}=[\tilde{\om}]_{\tilde{f}\circ \phi^{-1}}=[\tilde{\om}]_{\tilde{f}}$.
\end{defn}

\begin{remark}
The above definition does not depend of the choice of the representatives $\om$ and $\tilde{\om}$ since a local diffeomorphism maps forms with zero algebraic restriction to $f$ to forms with zero algebraic restrictions to $\tilde{f}$.
\end{remark}

The method of algebraic restrictions is based on the following result.

\begin{theorem}(Theorem A in \cite{DJZ2})\label{teo35}
Let $f:(\KK,0)\to (\KK^{2n},0)$ be the germ of a quasi-homogeneous curve. If $\om_0,\om_1$ are germs of symplectic forms on $(\KK^{2n},0)$ with the same algebraic restriction to $f$ then there exists a germ of a diffeomorphism $\Phi:(\KK^{2n},0)\to (\KK^{2n},0)$ such that $\Phi\circ f=f$ and $\Phi^{*}\om_1=\om_0$.
\end{theorem}


\begin{corollary}\label{corteoa}
The quasi-homogeneous curve-germs $f_1$ and $f_2$ are symplectomorphic in the symplectic space $(\KK^{2n},\om)$ if and only if the algebraic restrictions $[\om]_{f_1}$ and $[\om]_{f_2}$ are diffeomorphic.
\end{corollary}



Theorem \ref{teo35} reduces the symplectic classification of curve-germs which are diffeomorphic to some fixed quasi-homogeneous curve-germ to the classification of algebraic restrictions with symplectic representative to this fixed curve-germ.

Let $f$ be a parameterized curve-germ. Every algebraic restriction to $f$ of a closed 2-form has a symplectic representative due to the following result.

\begin{proposition}(\cite{DJZ2})\label{mergsymp}
Let $r$ be the minimal dimension of germs of smooth submanifolds of $(\KK^{2n},0)$ containing $\im(f)$. Let $(S,0)$ be one of such germs of dimension $r$. Let $\theta$ be the germ of a closed 2-form on $(\KK^{2n},0)$. Then there exists a germ of a symplectic form $\om$ on $(\KK^{2n},0)$ such that $[\theta]_f=[\om]_f$ if and only if $\rank({\theta}|_{T_0S})\geq 2r-2n$.
\end{proposition}

The following result shows the reason that this method is very powerful for analytic curves.

\begin{theorem}(\cite{D1})\label{teo2}
Let $C$ be the germ of a $\KK$-analytic curve. Then the space of algebraic restrictions to $C$ of germs of closed 2-forms is a finite-dimensional vector space.
\end{theorem}


\begin{theorem}(\cite{D1})\label{corteo2}
The space of algebraic restrictions of closed 2-forms to a parameterized quasi-homogeneous curve-germ $g$ is a finite-dimensional vector space spanned by algebraic restrictions of quasi-homogeneous closed 2-forms of bounded quasi-degrees.
\end{theorem}

Let $f:(\KK,0)\to (\KK^{2n},0)$ be a parametrized curve-germ and let $(M,0)$ be a germ of submanifold in $(\KK^{2n},0)$ such that $\im(f)\subset (M,0)$. Due to the following result we can replace the classification of algebraic restrictions to $f$ in $(\KK^{2n},0)$ to the classification of algebraic restrictions to $f$ in $(M,0)$.

\begin{proposition}(\cite{DJZ2})\label{prop3.6}
\begin{description}
	\item[(a)] Let $(M,0)$ be the germ of a smooth submanifold of $(\KK^m,0)$ containing $\im(f)$. Let $\om_1,\om_2$ be germs of $k$-forms on $(\KK^m,0)$.
Then $[\om_1]_f=[\om_2]_f$ if and only if $[\om_1|_{TM}]_f=[\om_2|_{TM}]_f$.
\item[(b)]
Let $f_1,f_2$ be curve-germs in $(\KK^m,0)$ whose images are contained in germs of equi-dimensional smooth submanifolds $(M_1,0), (M_2,0)$ respectively. Let $\om_1,\om_2$ be two germs of $k$-forms on $(\KK^m,0)$. The algebraic restrictions $[\om_1]_{f_1}$ and $[\om_2]_{f_2}$ are diffeomorphic if and only if the algebraic restrictions $[\om_1|_{TM_1}]_{f_1}$
and $[\om_2|_{TM_2}]_{f_2}$ are diffeomorphic.
\end{description}
\end{proposition}

The following two results are  very useful to obtain a basis of the space of algebraic restrictions of closed 2-forms to $f$.

\begin{proposition}(\cite{DJZ2})\label{basisclo}
Let $a_1,\ldots,a_k$ be a basis of the space of algebraic restrictions of 2-forms to a quasi-homogeneous parameterized curve-germ $f:(\KK,0)\to (\KK^{m},0)$  satisfying the following conditions:
\begin{description}
  \item[(1)] $da_1=\cdots=da_j=0$,
  \item[(2)] the algebraic restrictions $da_{j+1},\ldots,da_k$ are linearly independent.
\end{description}

Then $a_1,\ldots,a_j$ is a basis of the space of algebraic restrictions of closed 2-forms to $f$.
\end{proposition}

Fix a coordinate system  $(x_1,\ldots,x_m)$ in $(\KK^m,0)$.

\begin{lemma}(\cite{D1})\label{lemma6.2}
Let $f$ be a quasi-homogeneous curve on $(\KK^m,0)$. If the monomials $s(x)=\prod_{l=1}^{k}x_{l}^{s_l}$ and $p(x)=\prod_{l=1}^{k}x_{l}^{p_l}$ have the same quasi-degree then the forms $s(x)dx_i\wedge dx_j$ and $p(x)dx_i\wedge dx_j$ have the same algebraic restriction to $f$.
\end{lemma}

Let $\om$ be the germ of a $k$-form on $(\KK^m,0)$. We denote by $\om^{(r)}$ the quasi-homogeneous part of quasi-degree $r$ in the Taylor series of $\om$.

\begin{proposition}(\cite{D1})\label{prop6.5}
Let $f$ be a quasi-homogeneous curve in $(\KK^m,0)$. If $[\om]_f=0$ then $[\om^{(r)}]_f=0$ for any $r$.
\end{proposition}

This proposition allows us to define quasi-homogeneous algebraic restriction to a quasi-homogeneous parameterized curve-germ.

\begin{defn}
Let $a=[\om]_f$ be an algebraic restriction to $f$. The algebraic restriction $a^{(r)}=[\om^{(r)}]_f$ is called the \textbf{quasi-homogeneous part} of
\textbf{quasi-degree} $r$ of the algebraic restriction $a$; and $a$ is \textbf{quasi-homogeneous} of quasi-degree $r$ if $a=a^{(r)}$.
\end{defn}

\begin{defn}
Let $f$ be a curve in $(\KK^m,0)$. The germ $X$ of a vector field on $(\KK^m,0)$ is called \textbf{liftable} over $f$ if there exists a function germ $h$ on $(\KK,0)$ such that
\[
h \left( \frac{df}{dt} \right)=X\circ f.
\]
\end{defn}

The tangent space to the orbit of an algebraic restriction to $f$ $a$ is given by ${\mathcal L}_Xa$ for all vector fields $X$ liftable over $f$. The Lie derivative of an algebraic restriction with respect to a liftable vector field is well defined since $[{\mathcal L}_X\theta]_f=0$, for every $k$-form $\theta$ with zero algebraic restriction (Proposition 6.8 \cite{D1}).

Let $f$ be a quasi-homogeneous parameterized curve-germ and let $X$ be a smooth vector field. We denote by $X^{(r)}$ the quasi-homogeneous part of quasi-degree $r$ in the Taylor series of $X$. If $X$ is liftable over $f$ then $X^{(r)}$ is liftable over $f$ (Proposition 6.11 in \cite{D1}).

Let $K(f)$ be the minimal natural number such that all quasi-homogeneous algebraic restrictions to $f$ of closed 2-forms of quasi-degree greater than $K(f)$ vanish. By Theorem \ref{teo2}, $K(f)$ is finite.

\begin{theorem}(\cite{D1})\label{genvect}
Let $f(t)=(t^{\lambda_1},\ldots,t^{\lambda_k},0,\ldots,0)$. Let $X_s$ be the germ of a vector field such that $X_s\circ f=t^{s+1}df/dt$. Then the tangent space of the orbit of the quasi-homogeneous algebraic restriction $a_r$ of quasi-degree $r$ is spanned by ${\mathcal L}_{X_s}a_r$ for $s$ that are $\ZZ_{\geq 0}$-linear combinations of $\lambda_1,\ldots,\lambda_k$ and are smaller than $K(f)-r$.
\end{theorem}

We use the following result to obtain the normal forms of algebraic restrictions to a parameterized quasi-homogeneous curve-germ under the action of local symmetries.

\begin{theorem}(\cite{D1})\label{teo6.3}
Let $a_1,\ldots,a_p$ be a quasi-homogeneous basis of quasi-degrees $\delta_1\leq \ldots \leq \delta_s < \delta_{s+1} \leq \ldots \leq \delta_{p}$ of the space of algebraic restrictions of closed 2-forms to $f$. Let $a=\sum_{j=s}^{p}c_ja_j$, where $c_j\in \KK$ for $j=s,\ldots,p$ and $c_s\neq 0$. If there exists a liftable quasi-homogeneous vector field $X$ over $f$ such that ${\mathcal L}_Xa_s=ra_k$ for $k>s$ and $r\neq 0$ then $a$ is diffeomorphic to
$\sum_{j=s}^{k-1}c_ja_j+\sum_{j=k+1}^{p}b_ja_j$ for some $b_j\in \KK$, $j=k+1,\ldots,p$.
\end{theorem}
%
%
%
%


\section{Discrete symplectic invariants}

In \cite{D1}, \cite{DJZ2} and \cite{DT2} some symplectic invariants have been defined. They are the symplectic multiplicity, the index of isotropy and the Lagrangian tangency order. This system of invariants is not enough to distinguish the orbits in the classifications in this paper.
We recall these invariants and introduce a new invariant for algebraic restrictions to quasi-homogeneous parametrized curve-germs.

The first invariant is the symplectic multiplicity \cite{DJZ2} introduced in \cite{IJ} as the symplectic defect of a curve $f$.

\begin{defn}
The \textbf{symplectic multiplicity} ${\mu}_{sympl}(f)$ of a curve $f$ is the codimension of the symplectic orbit of $f$ in the $\mathcal A$-orbit of $f$.
\end{defn}

The symplectic multiplicity can be described in terms of algebraic restrictions.

\begin{proposition}(\cite{DJZ2})\label{sympm}
The symplectic multiplicity of a quasi-homogeneous curve $f$ in a symplectic space is equal to the codimension of the orbit of the algebraic restriction $[\om]_f$ with respect to the group of local diffeomorphisms preserving $f$ in the space of algebraic restrictions of closed 2-forms to $f$.
\end{proposition}

The second one is the index of isotropy.

\begin{defn}
The \textbf{index of isotropy} $\iota (f)$ of $f$ is the maximal order of vanishing of the 2-forms $\om |_{TS}$ over all smooth submanifolds $S$ containing $\im(f)$.
\end{defn}

Again we can describe the index of isotropy in terms of algebraic restrictions.

\begin{proposition}(\cite{DJZ2})\label{sympi}
The index of isotropy of a quasi-homogeneous curve $f$ in a symplectic space $(\KK^{2n},\om)$ is equal to the maximal order of vanishing of closed 2-forms representing the algebraic restriction $[\om]_f$
\end{proposition}

The Lagrangian tangency order was introduced in \cite{D1} following the ideas in \cite{A1}. If $H_1=\cdots =H_n=0$ define a smooth submanifold $L$ in the symplectic space $M$ then the tangency order of a curve $f:\KK \to M$ to $L$ is the minimum of the orders of vanishing at $0$ of the functions $H_1\circ f,\ldots, H_n\circ f$. We denote the tangency order of $f$ to $L$ by $t(f,L)$.

\begin{defn}
The \textbf{Lagrangian tangency order} $Lt(f)$ of a curve $f$ is the maximum of $t(f,L)$ over all smooth Lagrangian submanifolds $L$ of the symplectic space.
\end{defn}

The Lagrangian tangency order of a quasi-homogeneous curve in a symplectic space can also be expressed in terms of algebraic restrictions.

The order of vanishing of the germ of a 1-form $\alpha$ on a curve $f$ at 0 is the minimum of the orders of vanishing of the functions $\alpha(X)\circ f$ at 0 over all germs of smooth vector fields $X$. If $\alpha=\sum_{i=1}^{m} g_idx_i$ in local coordinates $(x_1,\ldots,x_m)$ then the order of vanishing of $\alpha$ on $f$ is the minimum of the orders of vanishing of functions $g_i\circ f$ for $i=1,\dots,m$.

\begin{proposition}(\cite{D1})\label{lto}
Let $f$ be the germ of a quasi-homogeneous curve such that the algebraic restriction of a symplectic form to it can be represented by a closed $2$-form vanishing at $0$. Then Lagrangian tangency order of the germ of a quasi-homogeneous curve $f$ is the maximum of the orders of vanishing on $f$ over all $1$-forms $\alpha$ such that $[\om]_f=[d\alpha]_f$.
\end{proposition}

Next we define an invariant to distinguish the different orbits of algebraic restrictions to quasi-homogeneous parameterized curve-germs of $k$-forms under the action of local symmetries. Let $f:(\KK,0)\to (\KK^{m},0)$ given by $f(t)=(t^{\lambda_1},\ldots,t^{\lambda_s},0,\ldots,0)$, where $\lambda_1<\cdots < \lambda_s$ are positive integers linearly independent over non-negative integers.

\begin{defn}
Let $a_1,a_2$ be algebraic restrictions to $f$ of $k$-forms on $(\KK^{m},0)$. Let $r_1,r_2$ be the smallest non-negative integers such that $a_{1}^{(r_1)}\neq 0$ and $a_{2}^{(r_2)}\neq 0$. We say that $a_1$ and $a_2$ have \textbf{proportional minimum quasi-degree part} if $r_1=r_2$ and
\[
a_2^{(r_1)}=ca_1^{(r_1)},
\]
for some $c\in \KK-\{0\}$.
\end{defn}

\begin{lemma}\label{prepmqdb}
Let $f:(\KK,0)\to (\KK^{s},0)$ be a curve-germ given by $f(t)=(t^{\lambda_1},\ldots,t^{\lambda_s})$, where $\lambda_1< \cdots <\lambda_s$ are positive integers linearly independent over non-negative integers. If $\Phi:(\KK^s,0)\to (\KK^s,0)$ is a local symmetry of $f$ then $\Phi$ is of the form
\begin{equation}\label{localsymmetrylp}
\begin{array}{l}
\Phi(x)=(c^{\lambda_1}x_1 +c_{1,2}x_2+\cdots + c_{1,s}x_s+\xi_1(x),\\
c^{\lambda_2}x_2 +c_{2,3}x_3+\cdots + c_{2,s}x_s+\xi_2(x),
\cdots,c^{\lambda_s}x_s+\xi_s(x)),
\end{array}
\end{equation}
where $c_{i,j}\in \KK$ and $\xi_i\in {\mathcal M}_s^{2}$.
\end{lemma}

\begin{proof}
If $\Phi:(\KK^s,0)\to (\KK^s,0)$ is a local symmetry of $f$ then there exists a germ of diffeomorphism $\phi:(\KK,0)\to (\KK,0)$ such that $\Phi\circ f=f\circ \phi$. Then
\[
\Phi(t^{\lambda_1},\ldots, t^{\lambda_s})=({\phi(t)}^{\lambda_1},\ldots,{\phi(t)}^{\lambda_s}).
\]
We can write $\phi(t)=ct+t^2u(t)$, where $u$ is smooth and $c\neq 0$. Then
\[
\Phi(t^{\lambda_1},\ldots,t^{\lambda_s})=(c^{\lambda_1}t^{\lambda_1}(1+\frac{tu(t)}{c})^{\lambda_1},\ldots,c^{\lambda_s}t^{\lambda_s}(1+\frac{tu(t)}{c})^{\lambda_s}).
\]
Since $\lambda_1< \cdots <\lambda_s$ are linearly independent over non-negative integers then $\Phi$ must be as in (\ref{localsymmetrylp}).
\end{proof}

\begin{theorem}\label{pmqdp}
Let $a_1,a_2$ be algebraic restrictions to $f$ of $k$-forms in $(\KK^m,0)$. If $a_1$ and $a_2$ are diffeomorphic then they have proportional minimum quasi-degree part.
\end{theorem}

\begin{proof}
We can suppose $f(t)=(t^{\lambda_1},\ldots,t^{\lambda_s})$ due to Proposition \ref{prop3.6}. Since $a_1$ and $a_2$ are diffeomorphic then there exist a local symmetry $\Phi:(\KK^s,0)\to (\KK^s,0)$ and $\phi:(\KK,0)\to (\KK,0)$ such that $\Phi\circ f=f\circ \phi$ and $\Phi^*a_1=a_2$. Due to Lemma \ref{prepmqdb} $\Phi$ must be as in (\ref{localsymmetrylp}).
%

Let $r_1,r_2$ be the smallest non-negative integers such that $a_{1}^{(r_1)}\neq 0$ and $a_{2}^{(r_2)}\neq 0$. Consider $\theta_1$ a $k$-form representing $a_1$. The $k$-form  $\theta_1^{(r_1)}$ can be assumed as
\[
\theta_1^{(r_1)}=\sum u_Ix_1^{\gamma_1 (I)}\ldots x_s^{\gamma_s (I)}dx_{i_1}\wedge \ldots \wedge dx_{i_k},
\]
where $I=i_1,\ldots,i_s$, $u_I\in \KK$ and $\gamma_1(I),\ldots,\gamma_s(I)$ are non-negative integers such that $x_1^{\gamma_1 (I)}\ldots x_s^{\gamma_s (I)}$ is a monomial with quasi-degree $r_1-\lambda_{i_1}-\cdots -\lambda_{i_k}$. Let $\eta=\theta_1-\theta_{1}^{(r_1)}$. We have
\[
a_2=\Phi^{*}a_1=\Phi^{*}[\theta_1]_f=[\Phi^{*}(\theta_1^{(r_1)}+\eta)]_f=[\Phi^{*}\theta_1^{(r_1)}]_f+[\Phi^{*}\eta]_f.
\]
One has $[\Phi^{*}\eta]_f^{(j)}=0$, for all $j<r_1+1$. Moreover
\[
\Phi^{*}\theta_1^{(r_1)}=c^{r_1}\theta_1^{(r_1)}+\nu,
\]
where $\nu$ is a $k$-form such that $\nu^{(j)}=0$, for all $j<r_1+1$. We have
\[
a_2=\Phi^{*}a_1=\Phi^{*}[\theta_1]_f=c^{r_1}[\theta_{1}^{(r_1)}]_f+[\mu]_f=c^{r_1}a_{1}^{r_1}+[\mu]_f,
\]
where $\mu=\Phi^*\eta + \nu$. Then $r_1=r_2$ and $a_2^{(r_1)}=c^{r_1}a_1^{(r_1)}$. Therefore $a_1$ and $a_2$ have proportional minimum quasi-degree part.
\end{proof}

\section{Symplectic singularities of curves with semigroup $(4,5,6,7)$}
The symplectic classification of curves with semigroup $(4,5,6,7)$ is given by the following Theorem.

\begin{theorem}\label{cla4,5,6,7}
\begin{description}
  \item[(i)] Let $(\RR^{2n}, \om_0 = \sum_{i=1}^{n} dp_i\wedge dq_i)$ be the symplectic space with the canonical coordinates $(p_1,q_1,\ldots,p_n,q_n)$. Then the germ of a curve $f:(\RR ,0)\to (\RR^{2n},0)$ with semigroup $(4,5,6,7)$ is symplectically equivalent to one and only one of the curves in the second column of Table \ref{r4} for $n=2$. When $n=3$ $f$ is symplectically equivalent to one and only one of the curves in the second column and rows 1-14 of Table \ref{r>4}. Finally when $n\geq 4$ $f$ is symplectically equivalent to one and only one of the curves in the second column of Table \ref{r>4}.
\item[(ii)] The parameters $c,c_1,c_2$ and $c_3$ are moduli.
\end{description}
\end{theorem}

\FloatBarrier
\begin{table}[H]
\caption{Normal forms in $\RR^{4}$}
\label{r4}
\[
\begin{array}{c l c c c }
\hline
& \text{Normal forms of}\, f & \mu_{sympl}(f) & \iota (f) & Lt(f)\\
\hline
1 & t\to(t^4,t^5+c_1t^7,t^6,c_2t^7), \,\ c_2\neq 0 & 2 & 0 & 5 \\

2 & t\to(t^4,\pm t^6+c_1t^7,t^5,c_2t^7), \,\ c_1,c_2\neq 0 & 3 & 0 & 6 \\
3 & t\to (t^4,\pm t^6,t^7,c_1t^5 +c_2t^6), \,\ c_1\neq 0 & 4 & 0 & 6 \\

4 & t\to (t^4,c_1t^7,t^6,-t^5+c_2t^7), \,\ c_1\neq 0,-\frac{3}{2},-\frac{1}{3},-1,\frac{12}{7} & 4 & 0 & 6 \\

5 & t\to (t^4,-\frac{3}{2}t^7,t^5,t^6+ct^7), \,\ c\neq 0 & 4 & 0 & 6 \\
6 & t\to (t^4,-\frac{3}{2}t^7,t^6,ct^7-t^5) & 5 & 0 & 6 \\

7 & t\to(t^4,-t^7+c_1t^{10},t^6,-t^5+c_2t^7) & 5 & 0 & 6 \\

8 & t\to(t^4,-\frac{1}{3}t^7+c_1t^{11},t^6,-t^5+c_2t^7) & 5 & 0 & 6 \\

9 & t\to (t^4,\frac{12}{7}t^7+ct^9,t^5,t^6) & 4 & 0 & 6 \\
\hline
\end{array}
\]
\end{table}

\FloatBarrier
\begin{table}[H]
\caption{Normal forms in $\RR^{2n}, n>2$}
\label{r>4}
\[
\begin{array}{c l c c c}
\hline
 & \text{Normal forms of}\, f & \mu_{sympl}(f) & \iota (f) & Lt(f)\\
\hline
1 & t\to(t^4,t^5 + c_1t^7,t^6,c_2t^7,t^7,0,\ldots,0) & 2 & 0 & 5 \\
2 & t\to(t^4,\pm t^6+c_1t^7,t^5,c_2t^7,t^7,0,\ldots,0), \,\ c_1\neq 0 & 3 & 0 & 6 \\
3 & t\to(t^4,\pm t^6,t^5,c_1t^7,t^6,c_2t^7,0,\ldots,0) & 4 & 0 & 6 \\

4 & t\to(t^4,c_1t^7,t^5,t^6,t^7,c_2t^6,0,\ldots,0),\, c_1\neq -\frac{3}{2},-\frac{1}{3},-1,\frac{12}{7} & 4 & 0 & 6 \\

5 & t\to(t^4,-\frac{3}{2}t^7,t^5,t^6+ct^7,0,\ldots,0), \,\ c\neq 0 & 4 & 0 & 6 \\
6 & t\to(t^4,-\frac{3}{2}t^7,t^6,ct^7-t^5,0,\ldots,0) & 5 & 0 & 6 \\

7 & t\to(t^4,-t^7+c_1t^{10},t^6,-t^5+c_2t^7,0,\ldots,0) & 5 & 0 & 6 \\

8 & t\to(t^4,-\frac{1}{3}t^7+c_1t^{11},t^6,-t^5+c_2t^7,0,\ldots,0) & 5 & 0 &  6\\

9 & t\to (t^4,\frac{12}{7}t^7+ct^9,t^5,t^6,0,\ldots,0) & 4 & 0 & 6 \\

10 & t\to(t^4,t^7,t^5,0,t^6,ct^7,0,\ldots,0) & 4 & 0 & 7 \\
11 & t\to(t^4,0,t^5,\pm t^7,t^6,ct^7,0,\ldots,0) & 5 & 0 & 7 \\
12 & t\to(t^4,ct^9,t^5,0,t^6,t^7,0,\ldots,0), c\neq -1,-2 & 6 & 0 & 7 \\
13 & t\to(t^4,2t^9+ct^{10},t^5,0,t^6,t^7,0,\ldots,0) & 7 & 0 & 7 \\
14 & t\to(t^4,-t^9+ct^{11},t^5,0,t^6,t^7,0,\ldots,0) & 7 & 0 & 7 \\
15 & t\to(t^4,t^9,t^5,0,t^6,0,t^7,0,\ldots,0) & 6 & 1 & 9 \\
16 & t\to(t^4,\pm t^{10},t^5,0,t^6,0,t^7,0,\ldots,0) & 7 & 1 & 10 \\
17 & t\to(t^4,t^{11},t^5,0,t^6,0,t^7,0,\ldots,0) & 8 & 1 & 11 \\
18 & t\to(t^4,0,t^5,0,t^6,0,t^7,0,\ldots,0) & 9 & \infty & \infty \\
\hline \vspace{0.6cm}
\end{array}
\]
\end{table}


The proof of Theorem \ref{cla4,5,6,7} is done in the following steps. Due to Proposition \ref{a-eq} $f$ is ${\mathcal A}$-equivalent to $g(t)=(t^4,t^5,t^6,t^7,0,\ldots,0)$. Due to Corollary \ref{corteoa}, the classification of parameterized curve-germs with semigroup $(4,5,6,7)$ under the symplectic action is given by the  classification of algebraic restrictions to $g$ of symplectic forms under the action of local symmetries of $g$. Due to Proposition \ref{mergsymp}, an algebraic restriction to $g$ of a closed 2-form has a symplectic representative in $\RR^{2n_0}$, for some $n_0\geq 2$. Proposition \ref{basis4,5,6,7} provides a basis for the vector space of algebraic restrictions of closed 2-forms to $g$.
In  Proposition \ref{claalg4,5,6,7} we find a finite set of orbits of algebraic restrictions to $g$ of symplectic forms. Finally, we apply Corollary \ref{corteoa} and Proposition \ref{claalg4,5,6,7}.
%
%


\begin{proposition}\label{basis4,5,6,7}
A basis of the vector space of algebraic restrictions of closed 2-forms to $g$
is given by:
\[
\begin{array}{l l l}
a_9=[dx_1\wedge dx_2]_g & a_{10}=[dx_1\wedge dx_3]_g & a_{11}^{+}=[dx_2\wedge dx_3]_g\\ \\
a_{11}^{-}=[dx_1\wedge dx_4]_g & a_{12}=[dx_2\wedge dx_4]_g & a_{13}^{+}=[dx_3\wedge dx_4]_g\\ \\
a_{13}^{-}=[x_1dx_1\wedge dx_2]_g  & a_{14}=[x_1dx_1\wedge dx_3]_g & a_{15}^{-}=[x_1dx_1\wedge dx_4]_g\\ \\
\end{array}
\]
\end{proposition}

\begin{proof}
This basis is obtained by the following algorithm on Singular software.

\begin{verbatim}
>LIB "all.lib";

>ring r=0,(x1,x2,x3,x4,t),dp;

>ideal J0=x1-t^4,x2-t^5,x3-t^6,x4-t^7;

>ideal J=eliminate(J0,t); // The ideal I(g)
consisting of polynomials vanishing on Im(g)

>ring q=0,(x1,x2,x3,x4,dx1,dx2,dx3,dx4),dp;

>setring q;

>ideal dm=dx1,dx2,dx3,dx4;

>ideal dm3=dm^3;

>def S=superCommutative(5,8,dm3); // The set of
differential forms identifying the 3-forms as 0.
Here we see S as a ring with the operations of
sum and exterior product

>setring S;

>proc extder (def p) // Defining the exterior derivative
on S

{
  def @t = typeof(p);
  if( (@t == "poly") or (@t == "vector") )
  {
    return (dx1*diff(p,x1)+dx2*diff(p,x2)+dx3*diff(p,x3)+
		dx4*diff(p,x4));
 }
  def @result = p;

  int @i = size(@result);

  if( (@t == "ideal") or (@t == "module") )
  {
    @i = ncols(@result);
  }

  for( ; @i > 0; @i-- )
  {
    @result [@i] = extder( @result [@i] );
  }
  return (@result);
};

>ideal Ig=std(imap(r,J)); // ideal I(g) in
S

>ideal dIg=extder(Ig); // ideal spanned by the
exterior derivative of the generators of I(g)


>ideal m=x1,x2,x3,x4;

>ideal dm=extder(m);

>ideal Z=std(Ig*dm^2+dIg*dm); // The set of
zero algebraic restrictions to g

>ideal c2=extder(m*dm); // the set of homogeneous
closed 2-forms of degree 2

>ideal c3=extder(m^2*dm); // the set of homogeneous
closed 2-forms of degree 3

>ideal c4=extder(m^3*dm); // the set of homogeneous
closed 2-forms of degree 4

>NF(c2,Z)+0; // elements of the basis (homogeneous)
of degree 2

>NF(c3,Z)+0; // elements of the basis (homogeneous)
of degree 3

>NF(c4,Z)+0; // verification that the vector space
of algebraic restrictions to g of closed 2-forms is
generated by NF(c2,Z) and NF(c3,Z).
\end{verbatim}
\end{proof}

\begin{proposition}\label{claalg4,5,6,7}
Let $a$ be an algebraic restriction to $g$ of a symplectic form in $(\RR^{2n},0)$, where $n\geq 2$. Then $a$ is diffeomorphic to one of the algebraic restrictions to $g$ of Table \ref{symm4} when $n=2$. For $n=3$ $a$ is diffeomorphic to one of the algebraic restrictions to $g$ in the rows 1-14 of Table \ref{symm6}. Finally, when $n\geq 4$ a is diffeomorphic to one of the algebraic restrictions to $g$ of Table \ref{symm6}.

\begin{table}[H]
\caption{Algebraic restrictions on $\RR^4$}
\label{symm4}
\[
\begin{array}{|l l|}
\hline
{[\om_1]_g}= & a_9+c_1a_{11}^{-} + c_{2}a_{13}^{+}, \,\ c_2 \neq 0 \\
\hline
{[\om_2]_g}= & \pm a_{10}+c_1a_{11}^{-}+c_2a_{12},\,\ c_1,c_2\neq 0\\

{[\om_3]_g}= & \pm a_{10}+c_1a_{12}+c_2a_{13}^{+},\,\ c_1\neq 0\\
\hline
{[\om_4]_g}= & a_{11}^{+}+c_1a_{11}^{-} + c_2a_{13}^{+},\,\ c_1\neq 0, -\frac{3}{2},-\frac{1}{3},-1,\frac{12}{7}\\
{[\om_5]_g}= & a_{11}^{+} - \frac{3}{2}a_{11}^{-} + ca_{12}, \,\ c\neq 0\\

{[\om_6]_g}= & a_{11}^{+} - \frac{3}{2}a_{11}^{-} + ca_{13}^{+}\\
{[\om_7]_g}= & a_{11}^{+} -a_{11}^{-} + c_1a_{13}^{+} + c_2a_{14}\\
{[\om_8]_g}= & a_{11}^{+}-\frac{1}{3}a_{11}^{-}+c_1a_{13}^{+}+c_2a_{15}^{-}\\

{[\om_9]_g}= & a_{11}^{+}+\frac{12}{7}a_{11}^{-}+ca_{13}^{-} \\
\hline
\end{array}
\]
\end{table}

\FloatBarrier
\begin{table}[H]
\caption{Algebraic restrictions on $\RR^{\geq 6}$}
\label{symm6}
\[
\begin{array}{|l l|}
\hline
{[\om_1]_g}= & a_9+c_1a_{11}^{-} + \,c_{2}a_{13}^{+}\\
\hline
{[\om_2]_g}= & \pm a_{10}+c_1a_{11}^{-}+c_2a_{12}, \,\ c_1\neq 0\\

{[\om_3]_g}= & \pm a_{10}+c_1a_{12}+c_2a_{13}^{+}\\
\hline
{[\om_4]_g}= & a_{11}^{+}+c_1a_{11}^{-} + c_2a_{13}^{+},\,\ c_1\neq -\frac{3}{2},-\frac{1}{3},-1,\frac{12}{7}\\
{[\om_5]_g}= & a_{11}^{+} - \frac{3}{2}a_{11}^{-} + ca_{12}, \,\ c\neq 0 \\

{[\om_6]_g}= & a_{11}^{+} - \frac{3}{2}a_{11}^{-} + ca_{13}^{+}\\

{[\om_7]_g}= & a_{11}^{+} -a_{11}^{-} + c_1a_{13}^{+} + c_2a_{14}\\
{[\om_8]_g}= & a_{11}^{+} - \frac{1}{3}a_{11}^{-} +c_{1}a_{13}^{+} +c_2a_{15}^{-}\\
{[\om_9]_g}= & a_{11}^{+}+\frac{12}{7}a_{11}^{-}+ca_{13}^{-}\\
\hline
{[\om_{10}]_g}= & a_{11}^{-}+ ca_{13}^{+} \\
\hline
{[\om_{11}]_g}= & \pm a_{12}+c a_{13}^{+}\\
\hline
{[\om_{12}]_g}= & a_{13}^{+}+ca_{13}^{-},c\neq -1,2\\
{[\om_{13}]_g}= & a_{13}^{+}+2a_{13}^{-}+ca_{14}\\
{[\om_{14}]_g}= & a_{13}^{+}-a_{13}^{-}+ca_{15}^{-}\\
{[\om_{15}]_g}= & a_{13}^{-}\\
\hline
{[\om_{16}]_g}= & \pm a_{14}\\
\hline
{[\om_{17}]_g}= & a_{15}^{-}\\
\hline
{[\om_{18}]_g}= & 0\\
\hline
\end{array}
\]
\end{table}

\end{proposition}

\begin{proof}
By Theorem \ref{genvect} we have to consider  vector fields $X_s$ such that $X_s\circ g=t^{s+1}dg/dt$ for $s=0,\ldots,6$. They are
\[
\begin{array}{l}
X_0=4x_1\frac{\partial}{\partial x_1}+5x_2\frac{\partial}{\partial x_2}+6x_3\frac{\partial}{\partial x_3}+7x_4\frac{\partial}{\partial x_4}\\[0.2cm]
X_1=4x_2\frac{\partial}{\partial x_1}+5x_3\frac{\partial}{\partial x_2}+6x_4\frac{\partial}{\partial x_3}+7x_1^{2}\frac{\partial}{\partial x_4}\\[0.2cm]
X_2=4x_3\frac{\partial}{\partial x_1}+5x_4\frac{\partial}{\partial x_2}+6x_1^2\frac{\partial}{\partial x_3}+7x_1x_2\frac{\partial}{\partial x_4}\\[0.2cm]
X_3=4x_4\frac{\partial}{\partial x_1}+5x_1^2\frac{\partial}{\partial x_2}+6x_1x_2\frac{\partial}{\partial x_3}+7m_{10}\frac{\partial}{\partial x_4}\\[0.2cm]
X_4=4x_1^2\frac{\partial}{\partial x_1}+5{x_1}x_2\frac{\partial}{\partial x_2}+6m_{10}\frac{\partial}{\partial x_3}+7m_{11}\frac{\partial}{\partial x_4}\\[0.2cm]
X_5=4x_1x_2\frac{\partial}{\partial x_1}+5m_{10}\frac{\partial}{\partial x_2}+6m_{11}\frac{\partial}{\partial x_3}+7m_{12}\frac{\partial}{\partial x_4}\\[0.2cm]
X_6=4m_{10}\frac{\partial}{\partial x_1}+5m_{11}\frac{\partial}{\partial x_2}+6m_{12}\frac{\partial}{\partial x_3}+7m_{13}\frac{\partial}{\partial x_4},
\end{array}
\]
where $m_i$ is  monomial with quasi-degree $i$.

The next table provides the infinitesimal action of these vector fields when $m_{10}=x_1x_3,m_{11}=x_1x_4,m_{12}=x_1^3$ and $m_{13}=x_1^2x_2$ on the elements of the basis.
\begin{table}[H]
\caption{Infinitesimal actions on algebraic restrictions of closed 2-forms to $g$}
\label{liederivate}
\[
\begin{array}{|c|c c c c|}
\hline
{\mathcal L}_{X_i}a_j & a_9 & a_{10} & a_{11}^{+} & a_{11}^{-}\\
\hline
X_0=E & 9a_9 & 10a_{10} & 11a_{11}^{+} & 11a_{11}^{-}\\
X_1 & 5a_{10} & 4a_{11}^{+} + 6a_{11}^{-} & 6a_{12} & 4a_{12}\\
X_2 & 5a_{11}^{-}-4a_{11}^{+} & 0 & -5a_{13}^{+}-12a_{13}^{-} & 4a_{13}^{+}+7a_{13}^{-}\\
X_3 & -4a_{12} & -4a_{13}^{+}+6a_{13}^{-} & 7a_{14} & 7a_{14} \\
X_4 & 13a_{13}^{-} & 14a_{14} & 5a_{15}^{-} & 15a_{15}^{-}\\
X_5 & 7a_{14} & 10a_{15}^{-} & 0 & 0 \\
X_6 & 5a_{15}^{-} & 0 & 0 & 0 \\
\hline
\end{array}
\]
\end{table}

\[
\begin{array}{|c|c c c c c|}
\hline
{\mathcal L}_{X_i}a_j & a_{12} & a_{13}^{+} & a_{13}^{-} & a_{14} & a_{15}^{-}\\
\hline
X_0 & 12a_{12} & 13a_{13}^{+} & 13a_{13}^{-} & 14a_{14} & 15a_{15}^{-}\\
X_1 & 5a_{13}^{+}-14a_{13}^{-} & -14a_{14} & 7a_{14} & 10a_{15}^{-} & 0 \\
X_2 & -\frac{7}{2}a_{14} & 5a_{15}^{-} & 5a_{15}^{-} & 0 & 0 \\
X_3 & 10a_{15}^{-} & 0 & 0 & 0 & 0 \\
X_4 & 0 & 0 & 0 & 0 & 0 \\
X_5 & 0 & 0 & 0 & 0 & 0 \\
X_6 & 0 & 0 & 0 & 0 & 0 \\
\hline
\end{array} \vspace{0.5cm}
\]

Due to Proposition \ref{mergsymp} we can write:
\[
a=t_1a_9+t_2a_{10}+t_3a_{11}^{+}+t_4a_{11}^{-}+t_5a_{12}+t_6a_{13}^{+}+t_7a_{13}^{-}+t_8a_{14}+t_9a_{15}^{-}
\]
where $t_i\in \RR$, $i=1,\ldots,9$. Suppose $t_1\neq 0$. Consider $\Psi_{1}:(\RR^4,0)\to (\RR^4,0)$ the local symmetry of $g$ given by
\[
\Psi_1(x_1,x_2,x_3,x_4)=(t_1^{-\frac{4}{9}}x_1,t_1^{-\frac{5}{9}}x_2,t_1^{-\frac{6}{9}}x_3,t_1^{-\frac{7}{9}}x_4).
\]
We have $\Psi_{1}^{*}a_9=\frac{1}{t_1} a_9$. Thus we can consider $t_1=1$. By Theorem \ref{teo6.3} $a$ is diffeomorphic to $a_9+c_1a_{11}^{-}+c_2a_{13}^{+}$, $c_1,c_2\in \RR$.

If $c_2\neq 0$, $a_9+c_1a_{11}^{-}+c_2a_{13}^{+}=[dx_1\wedge dx_2+c_1dx_1\wedge dx_4+c_2dx_3\wedge dx_4]_g$ and thus $a$ is diffeomorphic to $[\om_1]_g$ of Table \ref{symm4}. If $c_2=0$, $a_9+c_1a_{11}^{-}=[dx_1\wedge dx_2+c_1dx_1\wedge dx_4]_g$ and therefore is diffeomorphic to $[\om_1]_g$  of Table \ref{symm6}.

Suppose $t_1= 0$ and $t_2\neq 0$. Consider the local symmetry of $g$ $\Psi_2:(\RR^4,0)\to (\RR^4,0)$ given by
\[
\Psi_2(x_1,x_2,x_3,x_4)=(|t_2|^{-\frac{4}{10}}x_1,|t_2|^{-\frac{5}{10}}x_2,|t_2|^{-\frac{6}{10}}x_3,|t_2|^{-\frac{7}{10}}x_4).
\]
Then we can consider $t_2=\pm 1$. By Theorem \ref{teo6.3} $a$ is diffeomorphic to $\pm a_{10}+c_1a_{11}^{-}+c_2a_{12}+c_3a_{13}^{+}$, where $c_1,c_2,c_3\in \RR$. If $c_1\neq 0$ we prove in Lemma \ref{lemma4,5,6,7} that $a$ is diffeomorphic to $\pm a_{10}+c_1a_{11}^{-}+c_2a_{12}$. Thus $a$ is diffeomorphic to either $[\om_2]_g$ or $[\om_3]_g$ of Table \ref{symm4} if $a$ has a symplectic representative on $(\RR^4,0)$, or of Table \ref{symm6} if $a$ has a symplectic representative on $(\RR^{2n},0)$, $n\geq 3$.

Suppose $t_1=\cdots=t_{i-1}=0$ and $t_i\neq 0$, where $i\in \{ 4,5,7,8,9 \}$. Let $u$ be the smallest positive integer such that $a^{(u)}\neq 0$. Consider the local symmetries of $g$ $\Psi_i:(\RR^4,0)\to (\RR^4,0)$ defined by
\[
\begin{array}{l}
\Psi_i(x_1,x_2,x_3,x_4)=(t_i^{-\frac{4}{u}}x_1,t_i^{-\frac{5}{u}}x_2,t_i^{-\frac{6}{u}}x_3,t_i^{-\frac{7}{u}}x_4), \,\ \text{if} \,\ i\in \{ 4,7,9 \} \,\ \text{and}  \\[0.2cm]
\Psi_i(x_1,x_2,x_3,x_4)=(|t_i|^{-\frac{4}{u}}x_1,|t_i|^{-\frac{5}{u}}x_2,|t_i|^{-\frac{6}{u}}x_3,|t_i|^{-\frac{7}{u}}x_4), \,\ \text{if} \,\ i\in \{ 5,8 \}.
\end{array}
\]
Then we can consider $t_i=1$ if $i\in \{4,7,9 \}$ or $t_i =\pm 1$ if $i\in \{ 5,8\}$. Similarly, if $a$ has a symplectic representative on $(\RR^{6},0)$ then $a$ is diffeomorphic to either $[\om_{10}]_g$ or $[\om_{11}]_{g}$ of Table \ref{symm6} or if $a$ has a symplectic representative on $(\RR^{2n},0)$ for $n\geq 3$, then $a$ is diffeomorphic to $[\om_l]_g$ of Table \ref{symm6} for $l\in \{10,11,15,16,17 \}$.

Suppose $t_1=t_2=0$ and $t_3\neq 0$. As above we can assume $t_3=1$.
We apply Theorem \ref{teo6.3} for the following basis
\[
a_{9},a_{10},a_{11}^{-},a_{11}^{+}+t_4a_{11}^{-},a_{12},a_{13}^{+},a_{13}^{-},a_{14},a_{15}^{-}.
\]
Then $a=a_{11}^{+}+t_{4}a_{11}^{-}+t_5a_{12}+t_6a_{13}^{+}+t_{7}a_{13}^{-}+t_8a_{14}+t_9a_{15}^{-}$ is diffeomorphic to one of the following algebraic restrictions.
\begin{description}
	\item[(i)] $a_{11}^{+}+c_1a_{11}^{-}+c_2a_{13}^{+}, \,\ c_1\neq -\frac{3}{2},-1,-\frac{1}{3},\frac{12}{7}$;
	
	\item[(ii)] $a_{11}^{+}-\frac{3}{2}a_{11}^{-}+ca_{12}, \,\ c\neq 0$;
	
	\item[(iii)] $a_{11}^{+}-\frac{3}{2}a_{11}^{-}+ca_{13}^{+}$;
	
	\item[(iv)] $a_{11}^{+}-a_{11}^{-}+c_1a_{13}^{+}+c_2a_{14}$;
	
	\item[(v)] $a_{11}^{+}-\frac{1}{3}a_{11}^{-}+c_1a_{13}^{+}+c_2a_{15}^{-}$;
	
	\item[(vi)] $a_{11}^{+}+\frac{12}{7}a_{11}^{-}+ca_{13}^{-}$.
\end{description}

We shall see the cases (ii) and (iii) more carefuly. Suppose $a$ is of the form $a=a_{11}^{+}-\frac{3}{2}a_{11}^{-}+t_5a_{12}+t_6a_{13}^{+}+t_7a_{13}^{-}+t_8a_{14}+t_9a_{15}^{-}$. According to Theorem \ref{teo6.3}, $a$ is diffeomorphic to $a_{11}^{+}-\frac{3}{2}a_{11}^{-}+c_1a_{12}+c_2a_{13}^{+}$, $c_1,c_2\in \RR$. If $c_1\neq 0$, we prove in Lemma \ref{lemma4,5,6,7} that $a$ is diffeomorphic to $a_{11}^{+}-\frac{3}{2}a_{11}^{-}+ca_{12}$.

Then $a$ is diffeomorphic to $[\om_i]_g$, $i=4,\ldots,9$, of Table \ref{symm4}, if $a$ has a symplectic representative on $(\RR^4,0)$ or of Table \ref{symm6} if $a$ has a symplectic representative on $(\RR^{2n},0)$, $n\geq 3$.

When $t_i=0$, $i<6$ and $t_6\neq 0$ as above $a$ is diffeomorphic to $[\om_i]_g$ for $i=12,13,14$ of Table \ref{symm6}. Finally, when $t_i=0$, for all $i$, then $a=[dx_1\wedge dx_{n+1}+dx_2\wedge dx_{n+2}+\cdots +dx_n\wedge dx_{2n}]_g=[\om_{18}]_g$ of Table \ref{symm6}.
\end{proof}

\begin{lemma}\label{lemma4,5,6,7}
\begin{description}
	\item[(i)] The algebraic restriction $a=a_{11}^{+}-\frac{3}{2}a_{11}^{-}+c_1a_{12}+c_2a_{13}^{+}$, $c_1\neq 0$ is diffeomorphic to $[\om_5]_g$ of  Tables \ref{symm4} and \ref{symm6}.
 	
	\item[(ii)] The algebraic restriction $\tilde{a}=\pm a_{10}+c_1a_{11}^{-}+c_2a_{12}+c_3a_{13}^{+}$, $c_1\neq 0$ is diffeomorphic to $[\om_2]_g$ of  Tables \ref{symm4} and \ref{symm6}.
\end{description}
\end{lemma}

\begin{proof}
The proofs of the itens (i) and (ii) are very similar. We show (i).
We use Moser's homotopy method. Let $A_t= a_{11}^{+}-\frac{3}{2}a_{11}^{-}+c_1a_{12}+(1-t)c_2a_{13}^{+}$, where $t\in [0,1]$. Suppose the existence of a family of local symmetries of $g$ $\Phi_t:(\RR^4,0)\to (\RR^4,0)$, $t\in [0,1]$ such that

\begin{equation}\label{lemaposcla1}
{\Phi_t}^{*}A_t=a, \,\ \Phi_0=Id.
\end{equation}

Consider $\eta_t$ the vector field such that $\frac{d}{dt}\Phi_t=\eta_t\circ \Phi_t$. Differentiating (\ref{lemaposcla1}) we obtain
\[
{\mathcal L}_{\eta_t}A_t=c_2a_{13}^{+}.
\]
Suppose there exist smooth functions $b_i:[0,1]\to \RR$, $i=1,\ldots,4$, such that
\[
\eta_t=\displaystyle \sum_{i=1}^{4}b_i(t)X_i,
\]
where the $X_i's$ are vector fields of Table \ref{liederivate}, $i=1,\ldots,4$. We have
\[
\begin{array}{l l}
c_2a_{13}^{+}= & [5c_1b_1(t)-11b_2(t)]a_{13}^{+}+[-14c_1b_1(t)-\frac{45}{2}b_2(t)]a_{13}^{-}\vspace{0.2cm} \\
& +[-14(1-t)c_2b_1(t)-\frac{7}{2}c_1b_2(t)-\frac{7}{2}b_3(t)]a_{14}\vspace{0.2cm} \\
&[5(1-t)c_2b_2(t)+10c_1b_3(t)- \frac{35}{2}b_4(t)]a_{15}^{-}.
\end{array}
\]
The following system
\[
\left\{
\begin{array}{l}
5c_1b_1(t)-11b_2(t)=c_2 \vspace{0.2cm} \\
-14c_1b_1(t)-\frac{45}{2}b_2(t)=0 \vspace{0.2cm} \\
-14(1-t)c_2b_1(t)-\frac{7}{2}c_1b_2(t)-\frac{7}{2}b_3(t)=0 \vspace{0.2cm} \\
5(1-t)c_2b_2(t)+10c_1b_3(t)- \frac{35}{2}b_4(t)=0
\end{array}
\right.
\vspace{0.2cm}
\]
clearly has solution. As $\eta_t$ is liftable over $g$ the family of diffeomorphisms $\Phi_t$ associated to $\eta_t$ preserves $\im(g)$ and $\Phi_{t}^{*}A_t=a$, for all $t\in [0,1]$.
%
\end{proof}


\noindent \textbf{Proof of Theorem \ref{cla4,5,6,7}:} We use Tables \ref{symm4} and \ref{symm6} to obtain Tables \ref{r4} and \ref{r>4}.
 Let $\Phi:(\RR^{2n},0)\to (\RR^{2n},0)$ and $\phi:(\RR,0)\to (\RR,0)$ be germs of diffeomorphisms such that $\Phi\circ g=f\circ \phi$, where $g(t)=(t^4,t^5,t^6,t^7,0,\ldots,0)$. Consider $\sigma=\Phi^{*}\om$. Then
\[
\Phi^{*}[\om]_f=[\Phi^{*}\om]_{\Phi^{-1}\circ f}=[\sigma]_{g\circ \phi^{-1}}=[\sigma]_g.
\]

Due to Proposition \ref{claalg4,5,6,7},  $[\sigma]_g$ is diffeomorphic to one of the algebraic restrictions $[\om_i]_g$ of Tables \ref{symm4} and \ref{symm6}. Let $\rho_i:(\RR^{2n},0)\to (\RR^{2n},0)$ be the local symmetry of $g$ such that
\[
\rho_i^{*}[\sigma]_g=[\om_i]_g.
\]
By the classical Darboux Theorem there exists a germ of diffeomorphism $\Psi_i:(\RR^{2n},0)\to (\RR^{2n},0)$ such that $\Psi_{i}^{*}\om_i=\om$. Then
\[
(\Phi\circ \rho_i \circ \Psi_i)^{*}[\om]_f= \Psi_{i}^{*}\circ \rho_i^{*}\circ \Phi^{*}[\om]_f=[\om]_{\Psi_i^{-1}\circ g},
\]
that is, $[\om]_f$ and $[\om_i]_{\Psi_{i}^{-1}\circ g}$ are diffeomorphic. By Corollary \ref{corteoa}, $f$ and $\Psi_{i}^{-1}\circ g$ are symplectomorphic.

%
For example, the algebraic restriction $[\om_1]_g$ of Table \ref{symm4} is represented by
\[
\om_1=dx_1\wedge dx_2 + c_1dx_1\wedge dx_4+c_2dx_3\wedge dx_4= dx_1\wedge (dx_2+c_1dx_4)+dx_3\wedge (c_2dx_4),
\]
where $c_1,c_2\in \RR$ with $c_2\neq 0$.

Consider $F:(\RR^4,0)\to (\RR^4,0)$ defined by $F(x)=(x_1,x_2+c_1x_4,x_3,c_2x_4)$. Let $\Psi=F^{-1}$. Note that $\Psi^{*}\om_1=\om$. Therefore $f$ is symplectomorphic to
\[
 (\Psi^{-1}\circ g)(t)=(F\circ g)(t)=(t^4,t^5+c_1t^7,t^6,c_2t^7).
\]

The symplectic multiplicity and the index of isotropy follow straightforward from Theorem \ref{sympm} and Theorem \ref{sympi} respectively.
The Lagrangian tangency order can be calculated using Proposition \ref{lto} for the curves 15,16,17 and 18 from Table \ref{r>4}, the others curves we calculate by definition.

We observe that although some curves in Tables \ref{r4} and \ref{r>4}  do have the same symplectic invariants they are not symplectomorphic  by Proposition \ref{pmqdp} since the corresponding algebraic restrictions do not have proportional minimum quasi-degree part.

Now we prove that the parameters $c,c_1,c_2$ are moduli in the normal forms. The proofs are very similar in all cases. As an example we consider the orbit of $[\om_2]_g=\pm a_{10}+c_1a_{11}^{-}+c_2a_{12}$, $c_1\neq 0$ of Table \ref{symm4}. The tangent space to the orbit of $[\om_2]_g$  is spanned by the algebraic restrictions $\pm 10a_{10}+11c_1a_{11}^{-}+12c_2a_{12}, \pm (4a_{11}^{+}+6a_{11}^{-})+4c_1a_{12},a_{13}^{+},a_{13}^{-},a_{14},a_{15}^{-}$. One can easily see that  $a_{11}^{-}$ and $a_{12}$ are not in this tangent space. Therefore $c_1$ and $c_2$ are moduli.
\begin{flushright}
$\Box$
\end{flushright}


\section{Symplectic singularities of curves with semigroup $(4,5,6)$}

In this section we classify parameterized curve-germs with semigroup $(4,5,6)$. We follow the steps described in Section 4.

\begin{theorem}\label{class4,5,6}
\begin{description}
	\item[(i)] Let $(\RR^{2n},\om_0=\sum dp_i\wedge dq_i)$ be the symplectic space with canonical coordinates $(p_1,q_1,\ldots,p_n,q_n)$. Then the germ of a curve $f:(\RR,0)\to (\RR^{2n},0)$ with semigroup $(4,5,6)$ is symplectically equivalent to one and only one of the curves in the second column and rows 1-3 of Table \ref{4,5,6} for $n=2$. When $n>2$ $f$ is symplectically equivalent to one and only one of the curves in the second column of Table \ref{4,5,6}.
	
	\item[(ii)] The parameters are $c,c_1,c_2$ moduli.
\end{description}
\end{theorem}

\begin{table}[H]
\caption{Normal forms in $\RR^{2n}$}
\label{4,5,6}
\[
\begin{array}{c l c c c }
\hline
& \text{Normal forms of }\, f & \mu_{sympl}(f) & \iota (f) & Lt(f)\\
\hline
1 & t\to(t^4,t^5+c_1t^6,t^6,c_2t^5,0,\ldots,0) & 2 & 0 & 5 \\

2 & t\to(t^4,\pm t^6+c_1t^9,t^5,c_2t^6,0\ldots,0), \,\ c_1\neq 0 & 3 & 0 & 6 \\

3 & t\to (t^4,\pm t^6+c_1t^9,t^5,c_2t^{12},0,\ldots,0) & 4 & 0 & 6 \\

4 & t\to (t^4,c_1t^9+c_2t^{10},t^5,t^6,0,\ldots,0) & 4 & 0 & 6 \\

5 & t\to (t^4,t^9+c_1t^{10},t^5,0,t^6,c_2t^9,0,\ldots,0) & 5 & 1 & 9 \\

6 & t\to (t^4,\pm t^{10},t^5,c_1t^{12},t^6,c_2t^9,0,\ldots,0) & 6 & 1 & 10 \\

7 & t\to(t^4,0,t^5,0,t^6,-t^9+ct^{11},0,\ldots,0) & 6 & 1 & 11 \\

8 & t\to(t^4,0,t^5,ct^{14},t^6,-t^{11},0,\ldots,0) & 7 & 1 & 13 \\

9 & t\to (t^4,0,t^5,-t^{14},t^6,0,\ldots,0) & 7 & 2 & 15 \\

10 & t\to (t^4,0,t^5,0,t^6,0,\ldots,0) & 8 & \infty & \infty \\
\hline
\end{array}
\]
\end{table}

Let $g:(\RR,0)\to (\RR^{2n},0)$ given by $g(t)=(t^4,t^5,t^6,0,\ldots,0)$.
The proof of the next Proposition is similar to the proof of Proposition \ref{basis4,5,6,7}.
%

\begin{proposition}\label{base4,5,6}
A basis of the vector space of algebraic restrictions of closed 2-forms to $g$
is given by:
\[
\begin{array}{l l}
a_9=[dx_1\wedge dx_2]_g & a_{10}=[dx_1\wedge dx_3]_g\\[0.2cm]
a_{11}=[dx_2\wedge dx_3]_g & a_{13}=[x_1dx_1\wedge dx_2]\\[0.2cm]
a_{14}=[x_1dx_1\wedge dx_3]_g & a_{15}=[x_2dx_1\wedge dx_3+x_1dx_2\wedge dx_3]_g\\[0.2cm]
a_{17}=[x_3dx_2\wedge dx_3]_g & a_{19}=[x_2^{2}dx_1\wedge dx_2]_g
\end{array}
\]
\end{proposition}

\begin{proposition}\label{classalg456}
Let $a$ be an algebraic restriction to $g$ of a symplectic form in $(\RR^{2n},0)$, where $n\geq 2$. Then $a$ is diffeomorphic to one of the algebraic restrictions to $g$ in the rows 1-3 of Table \ref{symm4,5,6} for $n=2$. If $n\geq 3$ then $a$ is diffeomorphic to one of the algebraic restrictions to $g$ of Table \ref{symm4,5,6}.
\end{proposition}

\begin{table}[H]
\caption{Algebraic restrictions on $\RR^{2n}$}
\label{symm4,5,6}
\[
\begin{array}{|l l|}
\hline
{[\om_1]_g}= & a_9+c_1a_{10}+c_2a_{11} \\
\hline
{[\om_2]_g}= & \pm a_{10}+c_1a_{11}+c_2a_{13}, \,\ c_1\neq 0\\

{[\om_3]_g}= & \pm a_{10}+c_1a_{13}+c_2a_{17} \\
\hline
{[\om_4]_g}= & a_{11}+c_1a_{13}+c_2a_{14}\\
\hline
{[\om_5]_g}= & a_{13}+c_1a_{14}+c_2a_{15}\\
\hline
{[\om_6]_g}= & \pm a_{14}+c_1a_{15}+c_2a_{17} \\
\hline
{[\om_7]_g}= & a_{15}+ca_{17} \\
\hline
{[\om_8]_g}= & a_{17}+ca_{19}\\
\hline
{[\om_9]_g}= & a_{19}\\
\hline
[\om_{10}]_g =& 0\\
\hline
\end{array}
\]
\end{table}

\begin{proof}
Due to Theorem \ref{genvect} and Proposition \ref{base4,5,6} it is sufficient to consider liftable vector fields over $g$ of quasi-degree until 10.
\[
\begin{array}{l l}
X_0=E=4x_1\frac{\partial}{\partial x_1}+5x_2\frac{\partial}{\partial x_2}+6x_3\frac{\partial}{\partial x_3} &
X_4=x_1E \\[0.2cm]
X_5=x_2E &
X_6=x_3E \\[0.2cm]
X_7=4x_2x_3\frac{\partial}{\partial x_1}+5x_3^2\frac{\partial}{\partial x_2}+6x_1^2x_2\frac{\partial}{\partial x_3} &
X_8=x_1^2E \\[0.2cm]
X_9=x_1x_2E &
X_{10}=x_1x_3E \\[0.2cm]
\end{array}
\]

\FloatBarrier
\begin{table}[H]
\caption{Infinitesimal action on algebraic restrictions of closed 2-forms to $g$}
\label{liederivate4,5,6}
\[
\begin{array}{|c|c c c c c c c c|}
\hline
{\mathcal L}_{X_i}a_j & a_9 & a_{10} & a_{11} & a_{13} & a_{14} & a_{15} & a_{17} & a_{19}\\
\hline
X_0=E & 9a_9 & 10a_{10} & 11a_{11} & 13a_{13} & 14a_{14} & 15a_{15} & 17a_{17} & 19a_{19}\\
X_4 & 13a_{13} & 14a_{14} & 5a_{15} & -\frac{34}{3}a_{17} & 0 & 57a_{19} & 0 & 0\\
X_5 & 7a_{14} & 10a_{15} & 0 & 0 & 38a_{19} & 0 & 0 & 0 \\
X_6 & 5a_{15} & 0 & 17a_{17} & 19a_{19} & 0 & 0 & 0 & 0 \\
X_7 & 0 & 0 & 0 & 0 & 0 & 0 & 0 & 0\\
X_8 & -\frac{34}{3}a_{17} & 0 & 19a_{19} & 0 & 0 & 0 & 0 & 0\\
X_9 & 0 & 38a_{19} & 0 & 0 & 0 & 0 & 0 & 0 \\
X_{10} & 19a_{19} & 0 & 0 & 0 & 0 &0 & 0 & 0 \\
\hline
\end{array}
\]
\end{table}

Due to Propositions \ref{mergsymp} and \ref{base4,5,6} $a$ can be  written as:
\[
a=t_9a_9+t_{10}a_{10}+t_{11}a_{11}+t_{13}a_{13}+t_{14}a_{14}+t_{15}a_{15}+t_{17}a_{17}+t_{19}a_{19},
\]
where $t_i\in \RR$, $i\in J=\{9,10,11,13,14,15,17,19\}$.

Suppose $t_i\neq 0$, for some $i\in J$ and $t_u=0$, for all $u < i$. Considering the local symmetry $\Psi_i:(\RR^{3},0)\to (\RR^{3},0)$ defined by $\Psi_i(x)=(s_i^{-\frac{4}{i}}x_1,s_i^{-\frac{5}{i}}x_2,s_i^{-\frac{6}{i}}x_3)$, where $s_i=t_i$ if $i$ is odd and $s_i=|t_i|$ if $i$ is even,
we can suppose $t_i=1$ if $i$ is odd and $t_i=\pm 1$ if $i$ is even. By Theorem \ref{teo6.3} we conclude that $a$ is diffeomorphic to one of the following algebraic restrictions: \vspace{0.2cm}

\begin{multicols}{2}
\begin{enumerate}
\item $a_9+c_1a_{10}+c_2a_{11}$
\item $\pm a_{10}+c_1a_{11}+c_2a_{13}+c_3a_{17}$
\item $a_{11}+c_1a_{13}+c_2a_{14}$
\item $a_{13}+c_1a_{14}+c_2a_{15}$
\item $\pm a_{14}+c_1a_{15}+c_2a_{17}$
\item $a_{15}+ca_{17}$
\item $a_{17}+ca_{19}$
\item $a_{19}$
\end{enumerate}
\end{multicols}

If $t_i=0$, for all $i$, then $a=[\om_{10}]_g$. We prove in Lemma \ref{a10cut} that $\pm a_{10}+c_1a_{11}+c_2a_{13}+c_3a_{17}$ is diffeomorphic to $\pm a_{10}+c_1a_{11}+c_2a_{13}$ if $c_1\neq 0$.
\end{proof}

\begin{lemma}\label{a10cut}
Let $a=\pm a_{10}+c_1a_{11}+c_2a_{13}+c_3a_{17}$. If $c_1\neq 0$ then $a$ is diffeomorphic to $\pm a_{10}+{c}_1a_{11}+c_2a_{13}$.
\end{lemma}

\begin{proof}
We use the Moser's homotopy method. Let $A_t=\pm a_{10}+c_1a_{11}+c_2a_{13}+(1-t)c_3a_{17}$, where $t\in [0,1]$. Suppose there exists a family of local symmetries of $g$ $\Phi_t:(\RR^3,0)\to (\RR^3,0)$, $t\in [0,1]$, such that
\begin{equation}\label{456Phi}
\Phi_t^{*}A_t=a, \,\ \Phi_0=Id.
\end{equation}

Let $\eta$ be the vector field satisfying $\frac{d}{dt}\Phi_t=\eta \circ \Phi_t$. Differentiating (\ref{456Phi}) we obtain
\begin{equation}\label{eta456}
{\mathcal L}_{\eta}A_t=c_3a_{17}.
\end{equation}

Consider the vector field $\eta=\frac{c_3}{17c_1}X_6- \frac{c_2c_3}{17c_1^2}X_8$, where $X_6$ and $X_8$ are vector fields described in  Table \ref{liederivate4,5,6}.
Note that $\eta$ satisfies (\ref{eta456}). Then the family of diffeomorphisms $\Phi_t$  preserves $\im(g)$.
\end{proof}

\noindent \textbf{Proof of Theorem \ref{class4,5,6}}

The proof of Theorem \ref{class4,5,6} is similar to the proof of Theorem \ref{cla4,5,6,7}.
%

\begin{flushright}
$\Box$
\end{flushright}


\section{Classification of curves with semigroup $(4,5,7)$}

In this section we classify parameterized curve-germs with semigroup $(4,5,7)$. Again, we follow the steps as in Section 4.

\begin{theorem}\label{class4,5,7}
\begin{description}
	\item[(i)] Let $(\RR^{2n},\om_{0}=\sum dp_i\wedge dq_{i})$ be a symplectic space with canonical coordinates $(p_1,q_1,\ldots,p_n,q_n)$. Then the germ of a curve $f:(\RR,0)\to (\RR^{2n},0)$ with semigroup $(4,5,7)$ is symplectically equivalent to one and only one of the curves in the second column and rows 1-3 of Table \ref{4,5,7} if $n=2$. When $n>2$ $f$ is symplectically equivalent to one and only one of the curves in the second column of Table \ref{4,5,7}.
	
	\item[(ii)] The parameters $c,c_1,c_2$ are moduli.
\end{description}
\end{theorem}

\begin{table}[H]
\caption{Normal forms in $\RR^{2n}$}
\label{4,5,7}
\[
\begin{array}{c l c c c }
\hline
& \text{Normal forms of}\, f & \mu_{sympl}(f) & \iota (f) & Lt(f)\\
\hline
1 & t\to(t^4,t^5+ct^7,t^7,0,\ldots,0) & 1 & 0 & 5 \\

2 & t\to(t^4,t^7+c_1t^{9},t^5,c_2t^7,0\ldots,0) & 3 & 0 & 7 \\

3 & t\to (t^4,c_1t^9,t^5,\pm t^7+c_2t^9,0,\ldots,0) & 4 & 0 & 7 \\

4 & t\to (t^4,t^9+c_1t^{11},t^5,c_2t^9,t^6,0,\ldots,0) & 5 & 1 & 9 \\

5 & t\to (t^4,c_1t^{11},t^5,\pm t^9,t^7,c_2t^9,0,\ldots,0) & 6 & 1 & 10 \\

6 & t\to (t^4,t^{11},t^5,c_1t^{12},t^7,c_2t^9,0,\ldots,0) & 7 & 1 & 11 \\

7 & t\to(t^4,0,t^5,c_1t^{12},t^7,\pm t^9+c_2t^{11},0,\ldots,0) & 8 & 1 & 11 \\

8 & t\to(t^4,ct^{14},t^5,t^{12},t^{7},0,\ldots,0) & 8 & 1 & 13 \\

9 & t\to(t^4,0,t^5,0,t^7,\pm t^{11},0,\ldots,0) & 8 & 1 & 14 \\

10 & t\to (t^4,0,t^5,0,t^7,0,\ldots,0) & 9 & \infty & \infty \\
\hline
\end{array}
\]
\end{table}

Let $g:(\RR,0)\to (\RR^{2n},0)$ given by $g(t)=(t^4,t^5,t^7,0,\ldots,0)$.
%

\begin{proposition}\label{base4,5,7}
A basis of the vector space of algebraic restrictions of closed 2-forms to $g$
is given by:: \vspace{0.3cm}
\[
\begin{array}{l l}
a_9=[dx_1\wedge dx_2]_g & a_{11}=[dx_1\wedge dx_3]_g\vspace{0.2cm} \\
a_{12}=[dx_2\wedge dx_3]_g & a_{13}=[x_1dx_1\wedge dx_2]\vspace{0.2cm} \\
a_{14}=[x_2dx_1\wedge dx_2]_g & a_{15}=[x_1dx_1\wedge dx_3]_g\vspace{0.2cm} \\
a_{16}=[x_2dx_1\wedge dx_3+x_1dx_2\wedge dx_3]_g & a_{17}=[x_2dx_2\wedge dx_3]_g\vspace{0.2cm} \\
a_{18}=[x_3dx_1\wedge dx_3]_g &
\vspace{0.2cm}
\end{array}
\]
\end{proposition}

\begin{proposition}\label{classalg457}
Let $a$ be an algebraic restriction to $g$ of a symplectic form in $(\RR^{2n},0)$, where $n\geq 2$. Then $a$ is diffeomorphic to one of the algebraic restrictions to $g$ in the rows 1-3 of the Tables \ref{symm4,5,7}, for $n=2$. If $n>2$ then $a$ is diffeomorphic to one of the algebraic restrictions to $g$ of the Table \ref{symm4,5,7}.
\end{proposition}

\begin{table}[H]
\caption{Algebraic restrictions in $\RR^{2n}$}
\label{symm4,5,7}
\[
\begin{array}{|l l|}
\hline
{[\om_1]_g}= & a_9+ca_{11} \\
\hline
{[\om_2]_g}= & a_{11}+c_1a_{12}+c_2a_{13} \\
\hline
{[\om_3]_g}= & \pm a_{12}+c_1a_{13}+c_2a_{14} \\
\hline
{[\om_4]_g}= & a_{13}+c_1a_{14}+c_2a_{15} \\
\hline
{[\om_5]_g}= & \pm a_{14}+c_1a_{15}+c_2a_{16} \\
\hline
{[\om_6]_g}= & a_{15}+c_1a_{16}+c_2a_{17} \\
\hline
{[\om_7]_g}= & \pm a_{16}+c_1a_{17}+c_2a_{18} \\
\hline
{[\om_8]_g}= & a_{17}+ca_{18} \\
\hline
{[\om_9]_g}= & \pm a_{18}\\
\hline
[\om_{10}]_g= & 0\\
\hline
\end{array}
\]
\end{table}

\begin{proof}
Due to Theorem \ref{genvect} and Proposition \ref{base4,5,7}, it is sufficient to consider liftable vector fields over $g$ of quasi-degree until 9.
\[
\begin{array}{l l}
X_0=E=4x_1\frac{\partial}{\partial x_1}+5x_2\frac{\partial}{\partial x_2}+7x_3\frac{\partial}{\partial x_3} & \vspace{0.2cm} \\
X_3=4x_3\frac{\partial}{\partial x_1}+5x_1^2\frac{\partial}{\partial x_2}+7x_2^2\frac{\partial}{\partial x_3} & \vspace{0.2cm} \\
X_4=x_1E &
X_5=x_2E \vspace{0.2cm} \\
X_6=4x_2^2\frac{\partial}{\partial x_1}+5x_1x_3\frac{\partial}{\partial x_2}+7x_1^2x_2\frac{\partial}{\partial x_3} & \vspace{0.2cm} \\
X_7=x_3E &
X_8=x_1^2E \vspace{0.2cm}  \\
X_9=x_1x_2E & \vspace{0.2cm}  \\
\end{array}
\]

\FloatBarrier
\begin{table}[H]
\caption{Infinitesimal action on algebraic restrictions of closed 2-forms to $g$}
\label{liederivate4,5,7}
\[
\begin{array}{|c|c c c c c c c c c|}
\hline
{\mathcal L}_{X_i}a_j & a_9 & a_{11} & a_{12} & a_{13} & a_{14} & a_{15} & a_{16} & a_{17} & a_{18}\\
\hline
X_0=E & 9a_9 & 11a_{11} & 12a_{12} & 13a_{13} & 14a_{14} & 15a_{15} & 16a_{16} & 17a_{17} & 18a_{18}\\
X_3 & -4a_{12} & 14a_{14} & 10a_{15} & -4a_{16} & -\frac{17}{3}a_{17} & 18a_{18} & 0 & 0 & 0\\
X_4 & 13a_{13} & 15a_{15} & 12a_{16} & -\frac{17}{3}a_{17} & 18a_{18} & 0 & 0 & 0 & 0\\
X_5 & 14a_{14} & 4a_{16} & 17a_{17} & 18a_{18} & 0 & 0 & 0 & 0 & 0 \\
X_6 & 5a_{15} & \frac{17}{3}a_{17} & -9a_{18} & 0 & 0 & 0 & 0 & 0 & 0 \\
X_7 & -4a_{16} & 18a_{18} & 0 & 0 & 0 & 0 & 0 & 0 & 0\\
X_8 & -\frac{17}{3}a_{17} & 0 & 0 & 0 & 0 & 0 & 0 & 0 & 0\\
X_9 & 18a_{18} & 0 & 0 & 0 & 0 & 0 & 0 & 0 & 0\\
\hline
\end{array}
\]
\end{table}

Let $a$ be an algebraic restriction to $g$ of a symplectic form on $\RR^{2n}$, $n\geq 2$. Due to Propositions \ref{mergsymp} and \ref{base4,5,7}, $a$ can be written as
\[
a=t_9a_9+t_{11}a_{11}+t_{12}a_{12}+t_{13}a_{13}+t_{14}a_{14}+t_{15}a_{15}+t_{16}a_{16}+t_{17}a_{17}+t_{18}a_{18}.
\]
$t_i\in \RR$, for all $i\in \{ 9,11,12,13,14,15,16,17,18\}$.

Suppose $t_i\neq 0$, for some $i$. In this case, we consider the local symmetry $\Psi_i:(\RR^{3},0)\to (\RR^{3},0)$ defined by $\Psi_i(x)=(s_i^{-\frac{4}{i}}x_1,s_i^{-\frac{5}{i}}x_2,s_i^{-\frac{7}{i}}x_3)$, where $s_i=t_i$ if $i$ is odd, and $s_i=|t_i|$ if $i$ is even. Thus we can suppose $t_i=1$ if $i$ is odd or $t_i=\pm 1$ if $i$ is even. Applying the Theorem \ref{teo6.3} we conclude that $a$ is diffeomorphic to one of the algebraic restrictions in the rows 1-9 of the Table \ref{symm4,5,7}. If $t_i=0$, for all $i$, then $a=[\om_{10}]_g$.
\end{proof} \vspace{0.3cm}

\noindent \textbf{Proof the Theorem \ref{class4,5,7}}

The proof of this Theorem is similar to the proof of Theorem \ref{cla4,5,6,7}.

\begin{flushright}
$\Box$
\end{flushright}



\begin{thebibliography}{1}

\bibitem[A1]{A1}{ Arnold, V. I. First steps of local symplectic algebra, in: Differential topology, infinite-dimensional Lie algebras, and applications, D. B. Fuchs' 60th Anniversary Collection, American Mathematical Society Transl. Ser. 2, \emph{Am. Math. Soc.} 194(44) 1999, 1-8.}


\bibitem[AG]{AG}{Arnold, V. I.; Givental, A. B., Symplectic Geometry. \emph{Dynamical systems, IV}, 1-138, Encyclopedia of Mathematical Sciences, vol. 4, Springer, Berlin, 2001.}
%
%
%

\bibitem[BG]{BG}{Bruce, J. W.; Gaffney, T. J., Simple singularities of mappings $(\CC,0)\to (\CC^2,0)$. \emph{J. London Math. Soc.} 26 (1982) no. 3, 465-474.}
%

\bibitem[DGPS]{DGPS}{ Decker, W.; Greuel, G.-M.; Pfister, G.; Sch{\"o}nemann, H.:
\newblock {\sc Singular} {3-1-6} --- {A} computer algebra system for polynomial computations.
\newblock {http://www.singular.uni-kl.de} (2013).}

\bibitem[D1]{D1}{ Domitrz, W., Local symplectic algebra of quasi-homogeneous curves. \emph{Fund. Math.} 204 (2009) no. 1, 57-86}.

\bibitem[D2]{D2}{ Domitrz, W., Zero-dimensional symplectic isolated complete intersection singularities. \emph{J. Singul.} 6 (2012), 19-26.}


\bibitem[DJZ1]{DJZ1}{ Domitrz, W.; Janeczko, S.; Zhitomirskii, M.,
Relative Poincar\'{e} Lemma, contractibility, quasi-homogeneity and vector fields tangent to a singular variety,
\emph{Illinois J. Math.} 48 (2004), no. 3, 803-835.}

\bibitem[DJZ2]{DJZ2}{ Domitrz, W.; Janeczko S.; Zhitomirskii, M., Symplectic singularities of varieties:
the method of algebraic restrictions. \emph{J. Reine Angew. Math.} 618 (2008), 197-235.}

\bibitem[DR]{DR}{Domitrz, W.; Rieger, J. H., Volume preserving subgroups of $\mathcal A$ and $\mathcal K$ and singularities in unimodular geometry. \emph{Math. Ann.} 345 (2009), 783-817.}

\bibitem[DT1]{DT1}{Domitrz, W.; Tr\c{e}bska, $\dot{\text{Z}}$., Symplectic $S_\mu$ singularities, \emph{Real and Complex singularities, Contemp. Math. Soc.}, 569, Amer. Math. Soc., Providence, RI, (2012), 45-65.}

\bibitem[DT2]{DT2}{ Domitrz, W.; Tr\c{e}bska, $\dot{\text{Z}}$., Symplectic $T_7, T_8$ singularities and Lagrangian tangency orders. \emph{Proc. Edinb. Math. Soc. (2)} 55 (2012), no. 3, 657-683.}
%



\bibitem[IJ]{IJ}{ Ishikawa, G.; Janeczko, S., Symplectic bifurcations of plane curves and isotropy liftings, \emph{Quart. J. Math.} 54 (2003), 73-102.}

\bibitem[K]{K}{ Kolgushkin, P. A., Classification of simple multigerms of curves in a space with symplectic structure. \emph{St. Petersburg Math. J.} 15 (2003), 103-126.}




\bibitem[T1]{T1}{Tr\c{e}bska, $\dot{\text{Z}}$.,  Symplectic $W_8$ and $W_9$ singularities, \emph{J. Singul.} 6 (2012), 158-178.}

\bibitem[T2]{T2}{Tr\c{e}bska, $\dot{\text{Z}}$., Symplectic $U_7$, $U_8$ and $U_9$ singularities, \emph{Demonstr. Math.} 48 (2015), no. 2, 322-347.}

\end{thebibliography}
\end{document}